\documentclass[10pt, a4paper, reqno, twoside]{amsart}

\usepackage{amsmath, amsfonts, amsthm, amssymb, mathtools, multicol, scalefnt, relsize, mathbbol, graphicx, enumitem, xcolor, slashed, todonotes, xfrac, fullpage, autobreak, lipsum}
\usepackage[mathscr]{euscript}
\usepackage[T1]{fontenc}
\usepackage[latin1]{inputenc}
\usepackage[all]{xy}
\setlist[itemize]{itemsep=0pt, parsep=0pt, partopsep=0pt, topsep=1pt, leftmargin=30pt}

\mathtoolsset{showonlyrefs}

\usepackage[hypertexnames=false,
backref=page,
    pdftex,
    pdfpagemode=UseNone,
    breaklinks=true,
    extension=pdf,
    colorlinks=true,
    linkcolor=blue,
    citecolor=blue,
    urlcolor=blue,
]{hyperref}

\newcommand\bcdot{\ensuremath{
  \mathchoice {\mskip\thinmuskip\lower0.2ex\hbox{\scalebox{1.6}{$\cdot$}}\mskip\thinmuskip}}{\mskip\thinmuskip\lower0.2ex\hbox{\scalebox{1.6}{$\cdot$}}\mskip\thinmuskip}
   {\lower0.3ex\hbox{\scalebox{1.2}{$\cdot$}}}
   {\lower0.3ex\hbox{\scalebox{1.2}{$\cdot$}}}
}

\theoremstyle{plain}
\newtheorem{theorem}{Theorem}[section]
\newtheorem{lemma}[theorem]{Lemma}
\newtheorem{corollary}[theorem]{Corollary}

\theoremstyle{definition}

\newtheorem{remark}[theorem]{Remark}

\theoremstyle{plain}
\newtheorem{thmint}{Theorem}

\DeclareSymbolFontAlphabet{\mathbb}{AMSb}
\DeclareSymbolFontAlphabet{\mathbbl}{bbold}

\makeatletter
\@namedef{subjclassname@2020}{\textup{2020} Mathematics Subject Classification}
\makeatother

\allowdisplaybreaks[4]

\title[]{Regularizing estimates for positive solutions of the heat equation under geometric flows}

\author{Alessandro Goffi}
\address[Alessandro Goffi]{Dipartimento di Matematica e Informatica ``Ulisse Dini'' \\ Universit{\`a} di Firenze, viale Morgagni 67/a, 50134 Firenze, Italy}
\email{alessandro.goffi@unifi.it}

\author{Francesco Pediconi}
\address[Francesco Pediconi]{Dipartimento di Scienze Matematiche ``Giuseppe Luigi Lagrange'' \\ Politecnico di Torino, corso Duca degli Abruzzi 24, 10129 Torino, Italy}
\email{francesco.pediconi@polito.it}

\thanks{A.G.\ is member of INdAM--GNAMPA and F.P.\ is member of INdAM--GNSAGA. A.G.\ is partially supported by the INdAM--GNAMPA project 2026 ``Processi di diffusione non-lineari: regolarit{\`a} e classificazione delle soluzioni''.}

\begin{document}

\begin{abstract}
We study higher-order global estimates for the heat equation on Riemannian manifolds, both for static metrics and for metrics evolving under the Ricci flow. Under minimal geometric assumptions, we derive first-order regularizing estimates for log-solutions of the heat equation, together with upper second-order bounds with explicit constants. Our quantitative approach is based on integral duality methods proposed by L.\ C.\ Evans, J.-M.\ Lasry and P.-L.\ Lions in different settings.
\end{abstract}

\maketitle

\section{Introduction}
\setcounter{equation} 0

This note is devoted to the analysis of global estimates of \emph{positive solutions} $u(x,t) = u_t(x)$ to the heat equation $\partial_t u_t = \Delta u_t$ on a closed Riemannian manifold $M^n$ subject to bounds on the Ricci curvature and possibly to evolving metrics. In \cite[Theorem 1.1]{MR0834612}, Li--Yau proved
the following important semiconvexity-type estimate
$$
\Delta\log(u_t) = \frac{\Delta u_t}{u_t}-\frac{|\nabla u_t|^2}{u_t^2}\geq -\frac{n}{2t}
$$
under assumption $\mathrm{Ric} \geq 0$. This lower bound is a consequence of the maximum principle applied to suitable convex functions of $\nabla \log (u_t)$ obtained via the so-called \emph{Bernstein method}, and implies the parabolic Harnack inequality with explicit constants. This admits extensions to heat equations under the Ricci flow \cite{MR2601627} and other nonlinear variants for the pressure associated to the Porous Medium equation (cf.\ \cite{MR2487898}). Less is known about upper estimates, i.e. semiconcavity, on second derivatives of log-solutions, and, to our knowledge, they have been obtained when $0 < u_t \leq A$. These are typically achieved via maximum principle methods \cite{MR1230276,MR2285258,MR2317980} applied to functions involving $\nabla \log (u_t)$. Some results appeared in \cite{MR2365237, MR3472814} for compact manifolds and recently in \cite{GoffiTralli} for Neumann problems on unbounded convex domains of the Euclidean space. \smallskip

We introduce a new approach in the context of Riemannian manifolds to study differential Harnack-type estimates. The underlying idea is based on the observation that the function $v_t \coloneqq \log(u_t)$ solves a viscous Hamilton--Jacobi equation with purely quadratic \emph{uniformly concave} nonlinearity
$$
\partial_t v_t = \Delta v_t +|\nabla v_t|^2 \,\, .
$$
It is well-known that generalized solutions of this PDE satisfy first and second order regularizing effects driven by the nonlinear Hamiltonian and, remarkably, independent of the diffusion, see e.g. \cite{MR2597943,MR0667669}. In the case of a uniformly concave nonlinearity, semiconvexity estimates are natural and independent of the diffusion. On the contrary, the reversed semiconcavity estimates are less standard and depend both on the heat operator and on the regularity of the initial datum. The former bounds have generally been achieved either through the control-theoretic formulation of the problem or by the Bernstein method. To the best of our knowledge, there is no comprehensive analytic treatment of the latter estimates. Existing approaches include control-theoretic arguments \cite{MR2179357}, functional inequalities \cite{MR0450480}, constant rank theorems \cite{MR0792181,MR0829055,MR3941851}, maximum principle methods \cite{MR4771235}, doubling-of-variables methods \cite{MR2784332, MR3237760} or its probabilistic version based on coupling by reflection methods \cite{MR0841588} (see also \cite{Conforti} and the references in \cite{GoffiTralli}). Most of these results concern the preservation of concavity estimates. It is also known that, without regularity assumptions of the initial datum, log-semiconcavity estimates fail in general \cite{Mooney}. \smallskip

We exploit here a different viewpoint based on \textit{integration by parts techniques} that avoids as much as possible the maximum principle. The backbone of this strategy is the analysis of the formal adjoint of the linearization of the evolution of $v_t$, namely
$$
\partial_t \phi_t = \Delta\phi_t +2g(\nabla \phi_t, \nabla v_t) \,\, ,
$$
which in turn exhibits an advection-diffusion nature (see Section \ref{sec:1order}). This approach has been recently introduced by L.\, C.\, Evans \cite{MR2679366} to study quite different qualitative properties of the vanishing viscosity method for first-order fully nonlinear PDEs. The main peculiarities of the dual problem rely on its backward nature and the fact that the drift vector field driving the dynamics depends on the solution of the heat flow, giving rise to a coupled system of PDEs. Our analysis and estimates are heavily inspired, other than \cite{MR2679366}, by the recent theory of Mean Field Games, appearing in differential games with infinitely many rational agents and introduced by J.-M.\ Lasry and P.-L.\ Lions \cite{MR2295621}. In particular, we borrow some ideas to estimate crossed integrals involving the solution of the heat equation $u_t$ and the solution of the dual linearized problem $\rho_t$ \cite{MR4104825, MR3305653}. The recent paper \cite{GoffiTralli} exploited these integral methods to address first and second order estimates for the heat equation with homogeneous Neumann condition posed on unbounded convex domains of the Euclidean space. \medskip

We begin with a gradient estimate for positive solutions of the heat equation along a time-dependent Riemannian metric. More precisely, let $(M^n,g_t)$ be a closed manifold endowed with a smooth $1$-parameter family of Riemannian metrics, with $t\in[0,T]$, satisfying
\begin{equation} \label{eq:int-hypRfK}
\mathrm{Ric}_t +\frac{1}{2}\partial_t g_t \geq -K_0\,g_t
\quad \text{for all $t \in [0,T]$} \,\, ,
\end{equation}
for some $K_0 \geq 0$. Here and in what follows, time-dependent geometric quantities are denoted with a subscript $t$. Let also $u_t$ be a positive solution to the heat equation $\partial_t u_t = \Delta_t u_t$, with initial datum at $t=0$ satisfying $u_0 \leq A$, for some $A >0$. Then,
\begin{equation} \label{eq:int-1ord}
\big|\nabla^t\log(u_t)\big|_t^2 \leq \frac{2K_0}{1-e^{-2K_0t}}\,\log\!\left(\frac{A}{u_t}\right) \quad \text{for all $t \in (0,T]$}
\end{equation}
(see Theorem \ref{thm:1orderheat}). Notice that \eqref{eq:int-1ord} implies that $u_t \leq A$ for all $t \in [0,T]$, a fact that also follows directly from the maximum principle.

Estimate \eqref{eq:int-1ord} was already obtained in \cite[Theorem 1.2]{MR3845093}, under the assumption \eqref{eq:int-hypRfK}, via a maximum principle argument (see also \cite[Theorem 1.13]{MR3281843}). We remark that \eqref{eq:int-1ord} slightly improves the constant in the Hamilton--Souplet--Zhang estimates \cite{MR1230276, MR2285258} (see Remark \ref{rem:static}). In this paper, we obtain \eqref{eq:int-1ord} by a careful analysis of the evolution equation satisfied by the auxiliary function
$$
\psi_t \coloneqq \frac{1-e^{-2K_0t}}{2K_0} \, \big|\nabla^t\log(u_t)\big|_t^2 \,\, .
$$
We observe that \eqref{eq:int-1ord} yields a Liouville-type result for bounded positive solutions to the heat equation along \emph{ancient super Ricci flows}, namely smooth $1$-parameter families of Riemannian metrics $g_t$, defined for all $t < 0$, satisfying
$$
\partial_t g_t \geq -2\mathrm{Ric}_t
\quad \text{for all $t < 0$}
$$
(see Corollary \ref{cor:ancient}). This includes both ancient Ricci flows and static metrics with nonnegative Ricci curvature. Moreover, our integral method gives, as a byproduct, a new second order, two-sided $L^2$-estimate in terms of a lower bound for the Ricci curvature (see Estimate \eqref{eq:impHam*}). Note that, in the standard approach based on the maximum principle, the second order term is usually neglected because of its favourable sign. We then use these new integral bounds to derive second-order estimates for the solution $u_t$ in the same spirit of \cite{MR3472814} (see Theorem \ref{thm:2orderheat}). \smallskip

To present our results, let us focus here on the heat equation on $(M,g_t)$ in two significant cases, namely when $g_t=g$ is a static metric and when $g_t$ is a Ricci flow solution. The former reads as follows.

\begin{thmint} \label{thm:intstatic}
Let $(M^n,g)$ be a closed Riemannian manifold with $-K_0\,g \leq \mathrm{Ric} \leq K_1\,g$ for some $K_0, K_1 \geq 0$. Let $u_t$ be a positive solution to the heat equation $\partial_t u_t = \Delta u_t$, with initial datum at $t=0$ satisfying $u_0 \leq A$, for some $A >0$. Then
\begin{equation} \label{eq:int-static2ord}
\frac{\Delta u_t}{u_t} \leq 
\frac{n}{t} +\left(\frac{5K_0}{1-e^{-2K_0t}} +2K_1\right) \log\left(\frac{A}{u_t}\right) \quad \text{for all $t > 0$} \,\, .
\end{equation}
\end{thmint}

Compared to the previous results available in the literature, estimate \eqref{eq:int-static2ord} has the advantage of providing explicit constants that depend only on Ricci curvature bounds, and not on its derivatives (see Remark \ref{rem:int-static2ord}). Recent results have appeared in \cite{MR4963656, LuWuZhang, ChabiSouplet}. \smallskip

We remark that, when $g_t$ is time-dependent, some additional difficulties arise compared to the static metric case. In particular, the Laplacian $\Delta_t$ evolves with the metric, and so additional commutator terms appear when applying Bochner's identity to a space-time auxiliary function involving $\nabla^t \log(u_t)$. Furthermore, the dual linearized problem contains a zeroth-order term, due to the evolution of the metric. In the case of a Ricci flow solution, we obtain the following.

\begin{thmint} \label{thm:intRF}
Let $(M^n,g_t)$ be a Ricci flow solution on a closed Riemannian manifold, with $t \in [0,T]$, and let $u_t$ be a positive solution to the heat equation $\partial_t u_t = \Delta_t u_t$, with initial datum at $t=0$ satisfying $u_0 \leq A$, for some $A >0$. If $|\mathrm{Ric}_t|_t\leq K$ for all $t \in [0,T]$, for some $K \geq 0$, then
\begin{equation} \label{eq:int-RF2ord}
\frac{\Delta_t u_t}{u_t} \leq \left(\frac{n}{t}+ \frac{K}{\sqrt{2}}\right) +\left(\frac{5}{2t} +\left(2+\frac{1}{\sqrt{2}}\right)K\right)\log\!\left(\frac{A}{u_t}\right) \quad \text{for all $t \in (0,T]$} \,\, .
\end{equation}
\end{thmint}

We remark that estimate \eqref{eq:int-RF2ord} yields sharper constants than those in \cite[Theorem 1.3]{MR3472814}. Moreover, \eqref{eq:int-RF2ord} also holds when $g_t$ is a \emph{backward Ricci flow} solution, that is, $\partial_t g_t = +2\mathrm{Ric}_t$, with nonnegative Ricci curvature (see Corollary \ref{cor:HSZ2-backRF}).
\medskip

\noindent \emph{Plan of the paper}. Section \ref{sec:prel} lists some preliminary observations needed in the subsequent sections. Section \ref{sec:1order} provides regularizing gradient bounds for solutions of the heat equation under an evolving metric (see Theorem \ref{thm:1orderheat}). Section \ref{sec:2order} is divided into two subsections: Subsection \ref{sec:liyau} provides a different proof of the well-known Li--Yau estimate on compact manifolds under a lower Ricci curvature bound (see Theorem \ref{thm:liyau}), while Subsection \ref{sec:upper} is devoted to upper Laplacian estimates for solutions of the heat equation under an evolving metric (see Theorem \ref{thm:2orderheat}), from which both Theorem \ref{thm:intstatic} and Theorem \ref{thm:intRF} follow.
\medskip

\noindent \emph{Acknowledgements.\ } The authors would like to thank Philippe Souplet and Qi S.\ Zhang for useful comments on a first draft of the paper.

\section{Preliminaries}\label{sec:prel}
\setcounter{equation} 0

Let $M^n$ be an $n$-dimensional, closed manifold, and let $g_t$ be a smooth $1$-parameter family of Riemannian metrics such that
\begin{equation} \label{eq:ev-metric}
\partial_tg_t = -2h_t \,\, , \quad \text{for $t \in [0,T]$} \,\, ,
\end{equation}
where $h_t$ is a smooth $1$-parameter family of symmetric $(0,2)$-tensor fields on $M$. We denote by $D^t$ the Levi-Civita connection, by $|\cdot|_t$ the tensor norm, by $\nabla^t$ the gradient, by $\mathrm{div}_t$ the divergence, and by $\Delta_t \coloneqq + \mathrm{div}_t \circ \nabla^t$ the (negative) Laplace--Beltrami operator of $(M,g_t)$. Moreover, we denote by $\mathrm{Ric}_t$ and by $\mathrm{scal}_t$ the Ricci curvature and the scalar curvature of $(M,g_t)$, respectively, and we set
$$
H_t \coloneqq \mathrm{Tr}(g_t^{-1}h_t) \,\, .
$$
For the special case in which $g_t = g$ is independent of time, i.e.\ when $h_t = 0$, we will drop the $t$ dependence from the notation.

For later use, we recall the \emph{Bochner identity}, namely
\begin{equation} \label{eq:Boch}
\tfrac12\Delta_t|\nabla^t \phi|_t^2 =
g_t(\nabla^t(\Delta_t \phi),\nabla^t \phi)
+|(D^t)^2\phi|_t^2
+\mathrm{Ric}_t(\nabla^t \phi,\nabla^t \phi) \,\, ,
\end{equation}
for all smooth functions $\phi: M \to \mathbb{R}$. Moreover, if $\phi_t$ is a smooth $1$-parameter family of smooth functions $\phi_t : M \to \mathbb{R}$, then a direct computation shows that
\begin{equation} \label{eq:[d_t,nabla_t]}
g_t(\partial_t(\nabla^t\phi_t),X) = g_t(\nabla^t(\partial_t\phi_t),X) +2h_t(\nabla^t\phi_t,X)
\end{equation}
for every vector field $X$ on $M$.

We will always assume that $M$ is connected. Moreover, after passing to a double cover if necessary, we may also assume that $M$ is orientable. We then denote by $\mathrm{d}\mathrm{vol}_t$ the Riemannian volume form of $(M,g_t)$ and we recall that a direct computation in local coordinates shows that
\begin{equation} \label{eq:ev-vol}
\partial_t(\mathrm{d}\mathrm{vol}_t) = -H_t\,\mathrm{d}\mathrm{vol}_t \,\, .
\end{equation}
We also denote by $\delta_t$ the divergence of symmetric $(0,2)$-tensor fields on $(M,g_t)$, namely
$$
g_t(\delta_tT,X) \coloneqq -\sum_{i=1}^n (D^t_{e_i}T)(e_i,X) \,\, ,
$$
where $\{e_i\}$ is a local $g_t$-orthonormal frame for $TM$ and $X$ is a vector field on $M$. For later use, we prove the following result.

\begin{lemma}
Let $(M^n,g_t)$ be a closed manifold endowed with a smooth $1$-parameter family of Riemannian metrics satisfying \eqref{eq:ev-metric}. Then,
\begin{equation} \label{eq:derDelta}
(\partial_t\Delta_t)(\phi) =
2\,\mathrm{Tr}(g_t^{-1}h_t \circ g_t^{-1}(D^t)^2\phi)
-g_t(2\,\delta_t h_t +\nabla^tH_t,\nabla^t\phi)
\end{equation}
for every smooth function $\phi : M \to \mathbb{R}$.
\end{lemma}

\begin{proof}
Fix a choice of local coordinates on $M$, and denote by $g_{ij}$ the local components of the metric $g_t$, by $g^{ij}$ the entries of the inverse matrix $(g_{ij})^{-1}$, by $h_{ij}$ the local components of $h_t$, by $\partial_i$ the coordinate vector fields, by $D_i$ the Levi-Civita covariant derivative along $\partial_i$ and by $\Gamma_{ij}^k$ the Christoffel symbols. By definition, it follows that $H_t = g^{ij}h_{ij}$, while the local components of the vector field $\delta_th_t$ are $-g^{s\ell}g^{ij}D_ih_{j\ell}$. Since $\partial_t g_{ij} = -2 h_{ij}$, a direct computation shows that
\begin{gather}
\partial_t g^{ij} = 2\,g^{ir}g^{js}h_{rs} \,\, , \label{eq:coord1} \\
\partial_t \Gamma_{ij}^k = -g^{k\ell}(D_ih_{j\ell} +D_jh_{i\ell} -D_{\ell}h_{ij}) \,\, . \label{eq:coord2}
\end{gather}
Moreover, since
\begin{equation} \label{eq:coord3}
\Delta_t \phi = g^{ij}D_iD_j\phi = g^{ij}(\partial_i\partial_j\phi -\Gamma_{ij}^k\,\partial_k\phi) \,\, ,
\end{equation}
from \eqref{eq:coord1}, \eqref{eq:coord2} and \eqref{eq:coord3}, it follows that
$$\begin{aligned}
(\partial_t \Delta_t)\phi
&= 2\,g^{ir}g^{js}h_{rs}D_iD_j\phi +g^{ij}g^{k\ell}(D_ih_{j\ell} +D_jh_{i\ell} -D_{\ell}h_{ij})D_k\phi \\
&= 2\,g^{ir}g^{js}h_{rs}D_iD_j\phi
+2g^{ij}g^{k\ell}D_ih_{j\ell}D_k\phi
-g^{k\ell}D_{\ell}H_t\,D_k\phi \,\, .
\end{aligned}$$
Therefore, since
$$\begin{aligned}
\mathrm{Tr}(g_t^{-1}h_t \circ g_t^{-1}(D^t)^2\phi) &= g^{ir}g^{js}h_{rs}D_iD_j\phi \,\, ,\\
g_t(\delta_t h_t,\nabla^t\phi) &= -g^{ij}g^{k\ell}D_ih_{j\ell}D_k\phi \,\, ,\\
g_t(\nabla^tH_t, \nabla^t\phi) &= g^{k\ell}D_{\ell}H_t\,D_k\phi \,\, ,
\end{aligned}$$
equation \eqref{eq:derDelta} follows.
\end{proof}

\section{First-order regularizing effects of the heat equation under evolving metrics} \label{sec:1order}
\setcounter{equation} 0

We now prove the main result of the paper, which is the following Hamilton--Harnack inequality.

\begin{theorem} \label{thm:1orderheat}
Let $(M^n,g_t)$ be a closed manifold endowed with a smooth $1$-parameter family of Riemannian metrics satisfying \eqref{eq:ev-metric}, and let $K_0 \geq 0$ be such that
\begin{equation} \label{eq:hyp-hRic}
\mathrm{Ric}_{t} -h_t \geq -K_0\,g_t \quad \text{as bilinear forms for all $t \in [0,T]$} \,\, .
\end{equation}
Let $u_t$ be a positive solution to the heat equation
$$
\partial_t u_t = \Delta_t u_t \,\, , \quad t \in [0,T]
$$
with initial datum at $t=0$ satisfying $u_0 \leq A$, for some $A > 0$. Then,
\begin{equation} \label{eq:HSZ}
\big|\nabla^t\log(u_t)\big|_t^2 \leq \frac{2K_0}{1-e^{-2K_0t}}\,\log\!\left(\frac{A}{u_t}\right) \quad \text{for all $t \in (0,T]$} \,\, .
\end{equation}
Moreover, if $\rho_t$ is the positive solution to the backward equation
\begin{equation} \label{eq:eq-rho}
\partial_t\rho_t = -\Delta_t\rho_t +2\,\mathrm{div}_t\big(\rho_t\,\nabla^t \log(u_t)\big) +H_t\,\rho_t \,\, , \quad t \in [0,\tau]
\end{equation}
with final datum $\rho_{\tau} \in \mathcal{S}_{\tau}$ at $t = \tau$, where $\mathcal{S}_{\tau}$ is the set
\begin{equation} \label{eq:S}
\mathcal{S}_{\tau} \coloneqq \left\{f \in \mathcal{C}^{\infty}(M) : \int_{M}f\,\mathrm{d}\mathrm{vol}_t = 1 \, , \,\, f > 0 \right\} \,\, ,
\end{equation}
then
\begin{equation} \label{eq:impHam*}
\int_0^{\tau}\int_M \frac{1-e^{-2K_0t}}{2K_0}\,\big|(D^t)^2 \log(u_t)\big|_t^2\,\rho_t\,\mathrm{d}\mathrm{vol}_t\,\mathrm{d}t \leq \frac{1}{2}\int_M \log\left(\frac{A}{u_{\tau}}\right)\rho_{\tau}\,\mathrm{d}\mathrm{vol}_{\tau}
\end{equation}
for all $\tau \in [0,T]$ and $\rho_{\tau} \in \mathcal{S}_{\tau}$.
\end{theorem}

\begin{remark} \label{eu}
After reversing time, \eqref{eq:eq-rho} becomes a linear uniformly parabolic equation with smooth time-dependent coefficients. Since $M$ is closed, standard parabolic theory implies that, for every $\tau \in [0,T]$ and every final datum $\rho_{\tau} \in \mathcal{S}_{\tau}$, there exists a unique smooth positive solution $\rho_t$ to \eqref{eq:eq-rho} defined on $[0,\tau]$ such that $\rho_t|_{t = \tau} = \rho_{\tau}$. The fact that $\rho_t > 0$ for all $t\in[0,\tau]$ follows, e.g., from the comparison principle, cf.\ \cite[Lemma 5, p.\ 43]{MR0181836}, \cite[Section 8 and Proposition 8.1]{MR1780769} and \cite[Theorem 3.1, p.\ 145]{MR0241822}, \cite[Theorem 2.1.1]{MR2815949}. We emphasize that in the sequel of the paper, we are using that the drift $B_t = 2\,\nabla^t \log(u_t)$ is globally bounded. This comes from classical estimates for the heat equation. Therefore, the estimates \eqref{eq:HSZ} and \eqref{eq:impHam*} are conditional on the global boundedness of this vector field. However, our estimates do not depend quantitatively on this regularity and, therefore, it can be seen as a qualitative assumption which is always satisfied in our setting.
\end{remark}

\begin{proof}[Proof of Theorem \ref{thm:1orderheat}]
For the sake of notation, we set
\begin{equation} \label{eq:v-w-psi}
v_t \coloneqq \log(u_t) \,\, , \quad
w_t \coloneqq |\nabla^tv_t|_t^2 \,\, , \quad
\psi_t \coloneqq \mu(t)\,w_t,
\end{equation}
where $\mu:[0,T]\to \mathbb{R}$ is a smooth function to be chosen later, and we observe that a direct computation shows that
\begin{equation} \label{eq:ev-grad}
\partial_t v_t = \Delta_t v_t +|\nabla^tv_t|_t^2 \quad \text{for all $t \in [0,T]$} \,\, .
\end{equation}
Therefore, applying the Bochner identity \eqref{eq:Boch} to $v_t$ and using \eqref{eq:[d_t,nabla_t]}, \eqref{eq:v-w-psi} and \eqref{eq:ev-grad}, we compute
\begin{align}
\partial_t \psi_t
&= \mu'(t)w_t
-2\mu(t)\,h_t(\nabla^tv_t,\nabla^tv_t)
+2\mu(t)\,g_t(\partial_t(\nabla^tv_t),\nabla^tv_t) \nonumber \\
&= \mu'(t)w_t
-2\mu(t)\,h_t(\nabla^tv_t,\nabla^tv_t)
+2\mu(t)\,g_t(\nabla^t(\partial_tv_t),\nabla^tv_t)
+4\mu(t)\,h_t(\nabla^tv_t,\nabla^tv_t) \nonumber \\
&= \mu'(t)w_t
+2\mu(t)\,h_t(\nabla^tv_t,\nabla^tv_t)
+2\mu(t)\,g_t(\nabla^t(\Delta_t v_t),\nabla^tv_t)
+2\mu(t)\,g_t(\nabla^tw_t,\nabla^tv_t)
\nonumber \\
&= \Delta_t \psi_t +2g_t(\nabla^t\psi_t,\nabla^tv_t)
+\mu'(t)w_t
-2\mu(t)\,|D_t^2v_t|_t^2
-2\mu(t)\big(\mathrm{Ric}_t -h_t\big)(\nabla^tv_t,\nabla^tv_t) \,\, . \label{eq:ev-psi}
\end{align}
Consider then the time-dependent elliptic operator
$$
\mathcal{L}_t[\phi] \coloneqq \Delta_t \phi +g_t(\nabla^t\phi, B_t) \,\, ,
$$
where $B_t$ is a smooth, time-dependent vector field on $M$. Then, the conjugate operator $\mathcal{L}_t^*$ of $\mathcal{L}_t$ is
\begin{equation} \label{eq:L*}
\mathcal{L}_t^*[\rho] = -\Delta_t\rho +\mathrm{div}_t(\rho\,B_t) +H_t\,\rho \,\, .
\end{equation}
Indeed, if $\partial_t\phi_t = \mathcal{L}_t[\phi_t]$ and $\partial_t\rho_t = \mathcal{L}_t^*[\rho_t]$, then by \eqref{eq:ev-vol} and integrating by parts, we obtain
\begin{multline*}
\partial_t\left\{\int_M \phi_t \, \rho_t \,\mathrm{d}\mathrm{vol}_t\right\}
= \int_M \phi_t \, \mathcal{L}_t^*[\rho_t] \,\mathrm{d}\mathrm{vol}_t
+\int_M \phi_t\, \Delta_t\rho_t \,\mathrm{d}\mathrm{vol}_t\\
-\int_M \phi_t\mathrm{div}_t(\rho_t B_t) \,\mathrm{d}\mathrm{vol}_t
-\int_M \phi_t \,\rho_t \,H_t\,\mathrm{d}\mathrm{vol}_t \,\, .
\end{multline*}
Fix now $\tau \in [0,T]$, $\rho_{\tau} \in \mathcal{S}_{\tau}$ and consider the nonnegative solution of the backward equation
\begin{equation} \label{eq:ev-rho}
\left\{\begin{array}{l}
\partial_t \rho_t = \mathcal{L}^*_t[\rho_t] \\
\rho_t|_{t=\tau} = \rho_{\tau}
\end{array}\right. \,\, , \quad t \in [0,\tau] \,\, .
\end{equation}
We observe that, by \eqref{eq:ev-vol} and \eqref{eq:ev-rho}, the function $\rho_t$ verifies
$$
\partial_t\left\{\int_M \rho_t \,\mathrm{d}\mathrm{vol}_t\right\} =
-\int_M \Delta_t \rho_t \,\mathrm{d}\mathrm{vol}_t
+\int_M \mathrm{div}_t(\rho_t B_t) \,\mathrm{d}\mathrm{vol}_t
+\int_M H_t\,\rho_t \,\mathrm{d}\mathrm{vol}_t
-\int_M H_t\,\rho_t \,\mathrm{d}\mathrm{vol}_t
= 0
$$
and so
\begin{equation} \label{eq:mass}
\int_M \rho_t \,\mathrm{d}\mathrm{vol}_t = 1 \quad \text{for all $t \in [0,\tau]$} \,\, .
\end{equation}
Moreover, choosing $B_t = 2\,\nabla^tv_t$ and using \eqref{eq:hyp-hRic}, \eqref{eq:v-w-psi}, \eqref{eq:ev-psi} and \eqref{eq:ev-rho}, we get
\begin{align*}
\partial_t\left\{\int_M \psi_t \, \rho_t \,\mathrm{d}\mathrm{vol}_t\right\} &= \mu'(t)\int_M\rho_tw_t\,\mathrm{d}\mathrm{vol}_t
-2\mu(t)\int_M\rho_t|D_t^2v_t|_t^2\,\mathrm{d}\mathrm{vol}_t \\
&\qquad -2\mu(t)\int_M\rho_t\big(\mathrm{Ric}_t -h_t\big)(\nabla^tv_t,\nabla^tv_t)\,\mathrm{d}\mathrm{vol}_t \\
&\leq (\mu'(t)+2K_0\mu(t))\int_M\rho_tw_t\,\mathrm{d}\mathrm{vol}_t
-2\mu(t)\int_M\rho_t|D_t^2v_t|_t^2\,\mathrm{d}\mathrm{vol}_t \,\, .
\end{align*}
Therefore, by integrating in $t \in [0,\tau]$, since $\psi_0 =0$, we obtain 
\begin{equation} \label{eq:est1}
\int_M \psi_{\tau} \, \rho_{\tau} \,\mathrm{d}\mathrm{vol}_{\tau} 
+2\int_0^{\tau}\int_M \mu(t)\rho_t|D_t^2v_t|_t^2\,\mathrm{d}\mathrm{vol}_t\,\mathrm{d}t \leq \int_0^{\tau}\int_M (\mu'(t)+2K_0\mu(t))\rho_tw_t\,\mathrm{d}\mathrm{vol}_t\,\mathrm{d}t \,\, .
\end{equation}
We now set
$$
\mu(t) \coloneqq \frac{1-e^{-2K_0t}}{2K_0} \,\, , \quad \text{for $t \in [0,T]$} \,\, ,
$$
and we observe that
$$
\mu'(t)+2K_0\mu(t)=1 \quad \text{for all $t \in [0,T]$} \,\, , \quad \mu(0)=0 \,\, .
$$
Then,
\begin{equation} \label{eq:est1mu}
\int_M \psi_{\tau} \, \rho_{\tau} \,\mathrm{d}\mathrm{vol}_{\tau} 
+2\int_0^{\tau}\int_M\frac{1-e^{-2K_0t}}{2K_0}\rho_t|D_t^2v_t|_t^2\,\mathrm{d}\mathrm{vol}_t\,\mathrm{d}t \leq \int_0^{\tau}\int_M\rho_tw_t\,\mathrm{d}\mathrm{vol}_t\,\mathrm{d}t.
\end{equation}
By using \eqref{eq:ev-vol}, \eqref{eq:v-w-psi}, \eqref{eq:ev-grad}, \eqref{eq:ev-rho} and integrating by parts, we compute
\begin{align*}
\partial_t\left\{\int_M v_t \, \rho_t \,\mathrm{d}\mathrm{vol}_t\right\} &= 
\int_M\Delta_tv_t\,\rho_t\,\mathrm{d}\mathrm{vol}_t
+\int_M\rho_tw_t\,\mathrm{d}\mathrm{vol}_t
-\int_Mv_t\,\Delta_t\rho_t\,\mathrm{d}\mathrm{vol}_t
+2\int_Mv_t\,\mathrm{div}_t(\rho_t\nabla^tv_t)\,\mathrm{d}\mathrm{vol}_t \\
&\qquad +\int_Mv_t\,\rho_t\,H_t\,\mathrm{d}\mathrm{vol}_t
-\int_Mv_t\,\rho_t\,H_t\,\mathrm{d}\mathrm{vol}_t \\
&= -\int_M\rho_tw_t\,\mathrm{d}\mathrm{vol}_t \,\, .
\end{align*}
Therefore, since $u_0 \leq A$, from \eqref{eq:mass} we obtain
\begin{align}
\int_0^{\tau}\int_M\rho_tw_t\,\mathrm{d}\mathrm{vol}_t\,\mathrm{d}t &= \int_M v_0 \, \rho_0 \,\mathrm{d}\mathrm{vol}_0 -\int_M v_{\tau} \, \rho_{\tau} \,\mathrm{d}\mathrm{vol}_{\tau} \nonumber \\
&\leq \int_M \log(A) \, \rho_0 \,\mathrm{d}\mathrm{vol}_0 -\int_M \log(u_{\tau}) \, \rho_{\tau} \,\mathrm{d}\mathrm{vol}_{\tau} \nonumber \\
&= \int_M \log\left(\frac{A}{u_{\tau}}\right)\rho_{\tau}\,\mathrm{d}\mathrm{vol}_{\tau} \,\, . \label{eq:est2}
\end{align}
By \eqref{eq:v-w-psi}, \eqref{eq:est1} and \eqref{eq:est2}, we finally get
\begin{multline} \label{eq:final1}
\int_M \frac{1-e^{-2K_0\tau}}{2K_0}\,|\nabla^{\tau}v_{\tau}|_{\tau}^2\, \rho_{\tau} \,\mathrm{d}\mathrm{vol}_{\tau} 
+2\int_0^{\tau}\int_M\frac{1-e^{-2K_0t}}{2K_0}\,|(D^t)^2v_t|_t^2\,\rho_t\,\mathrm{d}\mathrm{vol}_t\,\mathrm{d}t \\
\leq \int_M \log\left(\frac{A}{u_{\tau}}\right)\rho_{\tau}\,\mathrm{d}\mathrm{vol}_{\tau} \,\, ,
\end{multline}
from which both \eqref{eq:HSZ} and \eqref{eq:impHam*} follow.
\end{proof}

\begin{remark} \label{rem:static}
To the best of our knowledge, the second order estimate \eqref{eq:impHam*} appears to be new. We also observe that assumption \eqref{eq:hyp-hRic} includes the case of static metrics, that is, when $h_t = 0$, with $\mathrm{Ric} \geq -K_0\,g$. Therefore, this generalizes Hamilton's original estimate \cite[Theorem 1.1]{MR1230276}, while it improves the constant of the a priori bound since
\begin{equation} \label{eq:est}
\frac{2K_0}{1-e^{-2K_0t}}\leq \frac{1}{t}+2K_0
\quad \text{for all $t>0$} \,\, .
\end{equation}
\end{remark}

\begin{remark}
Estimate \eqref{eq:HSZ} improves \cite[Theorem 3.1-(b)]{MR2250008}, which was obtained for the \emph{backward Ricci flow}, namely $h_t = -\mathrm{Ric}_t$, under the constraint $\mathrm{Ric}_t \geq 0$. The proof of estimate \eqref{eq:HSZ} provides an alternative approach to \cite[Theorem 1.1]{MR1230276}, giving the further integral bound \eqref{eq:impHam*} on second derivatives of log-solutions to the heat equation. We will use it to derive second order pointwise upper bounds on log-solutions in the next section. The previous proof formally extends to complete manifolds with lower bounds on the Ricci curvature: in this case \eqref{eq:HSZ-RF} agrees with the estimate of \cite{MR2317980} obtained by the maximum principle, while the $L^2$ second order estimate \eqref{eq:impHam*} is new to our knowledge. Following the above discussion, this can be seen as a $L^2$ estimate of second derivatives along the trajectories of the stochastic controlled dynamics driven by $B_t = 2\,\nabla^t \log(u_t)$ posed on a compact manifold. Notice also that we give an analytic proof of the bound obtained in \cite{MR3845093} via probabilistic methods.
\end{remark}

When $K_0 = 0$, Theorem \ref{thm:1orderheat} yields the following gradient estimate for the logarithm of positive solutions to the heat equation along \emph{super Ricci flows}, namely smooth $1$-parameter families of Riemannian metrics $g_t$ satisfying
\begin{equation} \label{eq:superRF}
\partial_tg_t \geq -2\mathrm{Ric}_t \,\, .
\end{equation}
This includes, in particular, the case of Ricci flow solutions, for which \eqref{eq:HSZ} was obtained by different methods in \cite[Theorem 3.2]{MR2250008}, and the case of static metrics with $\mathrm{Ric} \geq 0$ (see \cite[Theorem 1]{MR2317980} for the noncompact, complete case).

\begin{corollary} \label{cor:HSZ-RF}
Let $(M^n,g_t)$ be a super Ricci flow solution on a closed manifold,
i.e.\ a smooth $1$-parameter family of Riemannian metrics satisfying \eqref{eq:superRF}, with $t \in [0,T]$, and let $u_t$ be a positive solution to the heat equation $\partial_t u_t = \Delta_t u_t$, for $t \in [0,T)$, with initial datum at $t=0$ satisfying $u_0 \leq A$, for some $A >0$. Then,
\begin{equation} \label{eq:HSZ-RF}
\big|\nabla^t \log(u_t)\big|_t^2 \leq \frac{1}{t}\,\log\!\left(\frac{A}{u_t}\right) \quad \text{for all $t \in (0,T]$} \,\, .
\end{equation}
\end{corollary}

We observe that Corollary \ref{cor:HSZ-RF} yields Liouville-type results for bounded positive solutions to the heat equation along \emph{ancient} super Ricci flows, namely smooth $1$-parameter families of Riemannian metrics $g_t$ satisfying \eqref{eq:superRF} and defined for all $t \leq 0$.

\begin{corollary} \label{cor:ancient}
Let $(M^n,g_t)$ be an ancient super Ricci flow solution on a closed manifold. Then, every positive, bounded solution $u_t$ to the heat equation $\partial_t u_t = \Delta_t u_t$ defined for all $t \leq 0$ is constant.
\end{corollary}

\begin{proof}
By \eqref{eq:HSZ} and a time shift, for every $\alpha > 0$ it follows that
$$
|\nabla^t\log(u_t)|_t^2 \leq \frac{1}{\alpha+t}\,\log\!\left(\frac{A}{u_t}\right) \quad \text{for all $t \in (-\alpha,0]$} \,\, .
$$
Therefore, since $u_t$ is positive, by letting $\alpha \to +\infty$, one gets $\nabla^t u_t = 0$ for all $t \leq 0$. Since $u_t$ solves the heat equation, it follows that $\partial_t u_t = 0$ for all $t \leq 0$, and so $u_t$ is constant.
\end{proof}

\section{Explicit Laplacian estimates for the heat equation}
\label{sec:2order}
\setcounter{equation} 0

\subsection{The Li--Yau inequality under Ricci lower bounds} \label{sec:liyau} \hfill \par

We first discuss a new proof of \cite[Theorem 1.1]{MR0834612} under nonnegative Ricci curvature (see also \cite[Theorem 4.2, Theorem 4.3]{MR1333601} and \cite[Section 4.2]{MR2083636}). We begin with the following result.

\begin{theorem} \label{thm:liyau}
Let $(M^n,g)$ be a closed Riemannian manifold with $\mathrm{Ric} \geq -K_0\,g$, for some $K_0 \geq 0$, and let $u_t$ be a positive solution to the heat equation $\partial_t u_t = \Delta u_t$ with initial datum at $t=0$ satisfying $u_0 \leq A$, for some $A >0$. Then,
\begin{equation} \label{eq:LYinequalityK}
\Delta \log(u_t) \geq -\frac{n}{2t} -2K_0\log\left(\frac{A}{u_t}\right)
\quad \text{for all $t >0$} \,\, .
\end{equation}
\end{theorem}

\begin{proof}
For the sake of notation, we set
\begin{equation} \label{eq:v-w-varphi'}
v_t \coloneqq \log(u_t) \,\, , \quad
z_t \coloneqq \Delta v_t \,\, , \quad
\varphi_t \coloneqq t^{2} z_t \,\, .
\end{equation}
By using \eqref{eq:Boch}, \eqref{eq:ev-grad}, and \eqref{eq:v-w-varphi'}, we compute
$$\begin{aligned}
\partial_t z_t &= \Delta(\partial_t v_t) \\
&= \Delta z_t
+\Delta|\nabla v_t|^2 \\
&= \Delta z_t
+2g(\nabla z_t,\nabla v_t)
+2|D^2v_t|^2
+2\mathrm{Ric}(\nabla v_t,\nabla v_t)
\end{aligned}$$
and so $\varphi_t$ solves the equation
\begin{equation} \label{eq:ev-varphi'}
\partial_t \varphi_t = \Delta \varphi_t
+2g(\nabla \varphi_t,\nabla v_t)
+2tz_t
+2t^2|D^2v_t|^2
+2t^2\mathrm{Ric}(\nabla v_t,\nabla v_t) \,\, .
\end{equation}

Fix now $\tau \geq 0$ and let $\mathcal{S}_{\tau}$ be as in \eqref{eq:S}. Fix $\rho_{\tau} \in \mathcal{S}_{\tau}$ and consider the positive solution of the backward equation
\begin{equation} \label{eq:ev-rho''}
\left\{\begin{array}{l}
\partial_t \rho_t = \mathcal{L}^*_t[\rho_t] \\
\rho_t|_{t=\tau} = \rho_{\tau}
\end{array}\right. \,\, , \quad t \in [0,\tau] \,\, ,
\end{equation}
where $\mathcal{L}_t^*$ is given by \eqref{eq:L*}. By choosing again $B_t = 2\,\nabla^tv_t$, by using \eqref{eq:mass}, \eqref{eq:ev-varphi'}, \eqref{eq:ev-rho''}, and by integrating by parts, we obtain
\begin{equation} \label{eq:estintvarphi1'}
\partial_t\left\{\int_M \varphi_t \, \rho_t \,\mathrm{d}\mathrm{vol}\right\} \geq
\int_M (2t^2|D^2v_t|^2 +2t\Delta v_t -2t^2K_0|\nabla v_t|^2) \, \rho_t \,\mathrm{d}\mathrm{vol} \,\, .
\end{equation}
By the Cauchy--Schwarz inequality, we obtain
$$\begin{aligned}
2t^2|D^2v_t|^2 +2t\Delta v_t &\geq 2t^2|D^2v_t|^2 -2\sqrt{n}t|D^2v_t| \\
& = 2\Big(t|D^2v_t|-\frac{\sqrt{n}}{2}\Big)^2 -\frac{n}{2} \\
&\geq -\frac{n}{2} \,\, .
\end{aligned}$$
Therefore, by \eqref{eq:mass}, \eqref{eq:est2} and \eqref{eq:estintvarphi1'}, it follows that
$$\begin{aligned}
\int_M \tau^2 \Delta \log(u_{\tau})\, \rho_{\tau} \,\mathrm{d}\mathrm{vol} &\geq -
\int_0^{\tau}\int_M \frac{n}{2}\, \rho_t \,\mathrm{d}\mathrm{vol}\,\mathrm{d}t
-2K_0\int_0^{\tau}\int_M t^2\,|\nabla^tv_t|_t^2\,\rho_t\,\mathrm{d}\mathrm{vol}\,\mathrm{d}t \\
&\geq -\frac{n}{2}\,\tau -2K_0\tau^2\int_M \log\left(\frac{A}{u_{\tau}}\right)\rho_{\tau}\,\mathrm{d}\mathrm{vol} \\
&= -\int_M \tau^2\left(\frac{n}{2\tau} +2K_0\log\left(\frac{A}{u_{\tau}}\right)\right)\rho_{\tau} \,\mathrm{d}\mathrm{vol}
\end{aligned}$$
and this concludes the proof.
\end{proof}

\begin{remark}
The same argument used in the proof of Theorem \ref{thm:liyau} applies when $(M,g)$ is a compact Riemannian manifold with locally geodesically convex boundary, and the solution $u_t$ satisfies the homogeneous Neumann boundary conditions $\partial_{\nu}u_t=0$ on $\partial M$ for all $t>0$. Indeed, using the notation of the proof of Theorem \ref{thm:liyau} notice that
$$
z_t = \Delta v_t = \frac{\partial_t u_t}{u_t}-\frac{|\nabla u_t|^2}{u_t^2}
$$
and that
$$
\partial_\nu |\nabla u_t|^2\leq 0
$$
(see, e.g., \cite[Lemma 2.5]{MR4554482}). Therefore, this implies $\partial_{\nu} z_t \geq 0$, and the proof follows.
\end{remark}

In particular, when $K_0 = 0$, we recover the classical Li--Yau estimate.

\begin{corollary} \label{cor:liyau}
Let $(M^n,g)$ be a closed Riemannian manifold with $\mathrm{Ric} \geq 0$, and let $u_t$ be a positive solution to the heat equation $\partial_t u_t = \Delta u_t$. Then
\begin{equation} \label{eq:LYinequality}
\Delta \log(u_t) \geq -\frac{n}{2t}
\quad \text{for all $t >0$} \,\, .
\end{equation}
\end{corollary}

\begin{remark}
Some remarks on \eqref{eq:LYinequalityK} are in order. When $\mathrm{Ric} \geq -K_0\,g$, the geometric estimate found by Li--Yau states that
$$
\beta\,\partial_t \log (u_t)-|\nabla\log (u_t)|^2\geq -\frac{n\beta^2}{2t}-\frac{n\beta^2 K_0}{2(\beta-1)}
\quad \text{for all $\beta>1$ and $t>0$}
$$
(see, e.g., \cite[Corollary, p.\ 163]{MR1333601}). For $K_0 =0$, one can take $\beta = 1$, in which case the estimate is sharp, whereas for $K_0 >0$ the bound blows up as $\beta \to 1^+$. The problem of finding a sharp bound in the Li--Yau estimate has attracted considerable attention in recent years (see, e.g., \cite[Problem 10.5, p.\ 393]{MR2274812}). Several works have addressed the question of sharpening the bound by replacing $\beta$ with certain functions $\beta = \beta(t) > 1$ (see, e.g., \cite{MR3612336,MR4870314} and the references therein). The recent paper \cite{MR4870314} provided a Li--Yau estimate with $\beta = 1$ for compact manifolds using a Moser-type argument combined with the Hamilton--Souplet--Zhang estimate and heat kernel bounds. Our estimate in Theorem \ref{thm:liyau} provides a further bound in this direction using a different proof. 
\end{remark}

\subsection{Upper Laplacian estimates} \label{sec:upper} \hfill \par

We now boost the results of the previous section by proving a bound on the Laplacian of positive solutions of the heat equation under geometric flows. The main result of this section is the following.

\begin{theorem} \label{thm:2orderheat}
Let $(M^n,g_t)$ be a closed manifold endowed with a smooth $1$-parameter family of Riemannian metrics satisfying \eqref{eq:ev-metric}, and let $K_0, K_1 \geq 0$ be such that
\begin{equation} \label{eq:hyp-hRic2''}
\mathrm{Ric}_{t} -h_t \geq -K_0\,g_t
\,\, , \quad
\mathrm{Ric}_t \leq K_1\,g_t \quad
\text{as bilinear forms for all $t \in [0,T]$} \,\, .
\end{equation}
Assume also that there exist $K_2, K_3 \geq 0$ and $\frac{1}{2} < \alpha \leq \frac{3}{2}$ such that
\begin{equation} \label{eq:hyp-hRic2}
|h_t|_t \leq \frac{K_2}{(1+K_0t)^{\alpha}} \,\, , \quad
\big|2\,\delta_th_t+\nabla^tH_t\big|_t \leq \frac{K_3}{(1+t)^{\alpha}}
\quad \text{for all $t \in [0,T]$} \,\, .
\end{equation}
Let $u_t$ be a positive solution to the heat equation $\partial_t u_t = \Delta_t u_t$, for $t \in [0,T]$, with initial datum at $t=0$ satisfying $u_0 \leq A$, for some $A >0$. Then,
\begin{multline} \label{eq:Deltau/u}
\frac{\Delta_tu_t}{u_t} \leq 
\left(\frac{n}{t} +\frac{2K_2+K_3}{2\sqrt{2\alpha-1}}\right)
+\left(\frac{5K_0}{1-e^{-2K_0t}} +2K_1 +\frac{2K_2+K_3}{2\sqrt{2\alpha-1}}\right) \log\left(\frac{A}{u_t}\right)
\quad \text{for all $t \in (0,T]$} \,\, .
\end{multline}
\end{theorem}

\begin{proof}
We define
$$
\mu(t) \coloneqq \frac{1-e^{-2K_0t}}{2K_0} \,\, , \quad \text{for $t \in [0,T]$} \,\, .
$$
For the sake of notation, we set
\begin{equation} \label{eq:v-w-varphi}
v_t \coloneqq \log(u_t) \,\, , \quad
z_t \coloneqq \Delta_tv_t \,\, , \quad
\varphi_t \coloneqq t^2 z_t \,\, .
\end{equation}
By using \eqref{eq:Boch}, \eqref{eq:derDelta} and \eqref{eq:ev-grad}, we compute
$$\begin{aligned}
\partial_t z_t &= \Delta_t(\partial_t v_t) +(\partial_t \Delta_t)v_t \\
&= \Delta_t z_t
+\Delta_t|\nabla^tv_t|_t^2
+2\,\mathrm{Tr}(g_t^{-1}h_t \circ g_t^{-1}(D^t)^2v_t)
-g_t(2\,\delta_t h_t +\nabla^tH_t,\nabla^tv_t) \\
&= \Delta_t z_t
+2g_t(\nabla^tz_t,\nabla^tv_t)
+2|(D^t)^2v_t|_t^2
+2\mathrm{Ric}_t(\nabla^tv_t,\nabla^tv_t) \\
&\qquad +2\,\mathrm{Tr}(g_t^{-1}h_t \circ g_t^{-1}(D^t)^2v_t)
-g_t(2\,\delta_t h_t +\nabla^tH_t,\nabla^tv_t)
\end{aligned}$$
and so
\begin{multline} \label{eq:ev-varphi}
\partial_t \varphi_t = \Delta_t \varphi_t
+2g_t(\nabla^t\varphi_t,\nabla^tv_t)
+2tz_t
+2t^2|(D^t)^2v_t|_t^2
+2t^2\mathrm{Ric}_t(\nabla^tv_t,\nabla^tv_t) \\
+2t^2\mathrm{Tr}(g_t^{-1}h_t \circ g_t^{-1}(D^t)^2v_t)
-t^2g_t(2\,\delta_t h_t +\nabla^tH_t,\nabla^tv_t) \,\, .
\end{multline}
We also notice that, by \eqref{eq:hyp-hRic2''}, \eqref{eq:hyp-hRic2} and the Cauchy--Schwarz inequality, we have
\begin{equation} \label{eq:estimates} \begin{aligned}
2\mathrm{Ric}_t(\nabla^tv_t,\nabla^tv_t)
&\leq 2K_1\, |\nabla^tv_t|_t^2 \,\, , \\
2\mathrm{Tr}(g_t^{-1}h_t \circ g_t^{-1}(D^t)^2v_t)
&\leq \frac{2K_2}{(1+K_0t)^{\alpha}}\,|(D^t)^2v_t|_t \,\, , \\
-g_t(2\,\delta_t h_t +\nabla^tH_t,\nabla^tv_t) &\leq \frac{K_3}{(1+t)^{\alpha}}\,|\nabla^tv_t|_t \,\, .
\end{aligned} \end{equation}

Fix now $\tau \in [0,T]$ and let $\mathcal{S}_{\tau}$ be as in \eqref{eq:S}. Fix $\rho_{\tau} \in \mathcal{S}_{\tau}$ and consider the positive solution of the backward equation
\begin{equation} \label{eq:ev-rho'}
\left\{\begin{array}{l}
\partial_t \rho_t = \mathcal{L}^*_t[\rho_t] \\
\rho_t|_{t=\tau} = \rho_{\tau}
\end{array}\right. \,\, , \quad t \in [0,\tau] \,\, ,
\end{equation}
where $\mathcal{L}_t^*$ is given by \eqref{eq:L*}. By choosing again $B_t = 2\,\nabla^tv_t$, by using \eqref{eq:mass}, \eqref{eq:ev-varphi}, \eqref{eq:estimates} and \eqref{eq:ev-rho'}, and by integrating by parts, we obtain
\begin{align*} \label{eq:estintvarphi1}
\partial_t\left\{\int_M \varphi_t \, \rho_t \,\mathrm{d}\mathrm{vol}_t\right\} &\leq
2\int_M t\,z_t \, \rho_t \,\mathrm{d}\mathrm{vol}_t
+2\int_M t^2\,|(D^t)^2v_t|_t^2 \, \rho_t \,\mathrm{d}\mathrm{vol}_t \\
&+2K_1\int_M t^2\,|\nabla^tv_t|_t^2 \,\rho_t \,\mathrm{d}\mathrm{vol}_t
+\frac{2K_2}{(1+K_0t)^{\alpha}}\int_M t^2\,|(D^t)^2v_t|_t \, \rho_t \,\mathrm{d}\mathrm{vol}_t \\
&+\frac{K_3}{(1+t)^{\alpha}}\int_M t^2\,|\nabla^tv_t|_t \,\rho_t \,\mathrm{d}\mathrm{vol}_t \,\, .
\end{align*}
Let us observe now that, by \eqref{eq:v-w-varphi} and the Cauchy--Schwarz inequality, we have
$$
z_t \leq \sqrt{n}\,|(D^t)^2v_t|_t
$$
and so, since $\varphi_0 = 0$, we obtain
$$
\begin{aligned}
\int_M \varphi_{\tau} \, \rho_{\tau} \,\mathrm{d}\mathrm{vol}_{\tau} &\leq
2\int_0^{\tau}\int_M t^2\,\,|(D^t)^2v_t|_t^2 \, \rho_t \,\mathrm{d}\mathrm{vol}_t\,\mathrm{d}t
+2K_1\int_0^{\tau}\int_M t^2\,|\nabla^tv_t|_t^2 \,\rho_t \,\mathrm{d}\mathrm{vol}_t\,\mathrm{d}t \\
&\qquad +2K_2\int_0^{\tau}\int_M \frac{t^2}{(1+K_0t)^{\alpha}}\,|(D^t)^2v_t|_t \, \rho_t \,\mathrm{d}\mathrm{vol}_t\,\mathrm{d}t \\
&\qquad +2\sqrt{n}\int_0^{\tau}\int_M t\, |(D^t)^2v_t|_t \, \rho_t \,\mathrm{d}\mathrm{vol}_t\, \mathrm{d}t \\
&\qquad +K_3\int_0^{\tau}\int_M \frac{t^2}{(1+t)^{\alpha}}\,|\nabla^tv_t|_t \,\rho_t \,\mathrm{d}\mathrm{vol}_t\,\mathrm{d}t
\,\, .
\end{aligned}
$$
On the other hand, by Young's inequality and \eqref{eq:mass}, we obtain
$$\begin{aligned}
2\sqrt{n}\int_0^{\tau}\int_M t\,|(D^t)^2v_t|_t \, \rho_t \,\mathrm{d}\mathrm{vol}_t\, \mathrm{d}t &\leq n\int_0^{\tau}\mathrm{d}t + \int_0^{\tau}\int_M t^2\,|(D^t)^2v_t|_t^2\,\rho_t\,\mathrm{d}\mathrm{vol}_t\,\mathrm{d}t \\
&= n\tau +\int_0^{\tau}\int_M t^2\,|(D^t)^2v_t|_t^2\,\rho_t\,\mathrm{d}\mathrm{vol}_t\,\mathrm{d}t
\end{aligned}$$
and so
\begin{equation} \label{eq:estintvarphi2}
\begin{aligned}
\int_M \varphi_{\tau} \, \rho_{\tau} \,\mathrm{d}\mathrm{vol}_{\tau} &\leq
n \tau +3\int_0^{\tau}\int_M t^2\,|(D^t)^2v_t|_t^2 \, \rho_t \,\mathrm{d}\mathrm{vol}_t\,\mathrm{d}t +2K_1\int_0^{\tau}\int_M t^2|\nabla^tv_t|_t^2 \,\rho_t \,\mathrm{d}\mathrm{vol}_t\,\mathrm{d}t \\
&\qquad +2K_2\int_0^{\tau}\int_M \frac{t^2}{(1+K_0t)^{\alpha}}|(D^t)^2v_t|_t \, \rho_t \,\mathrm{d}\mathrm{vol}_t\,\mathrm{d}t \\
&\qquad +K_3\int_0^{\tau}\int_M \frac{t^2}{(1+t)^{\alpha}}|\nabla^tv_t|_t \,\rho_t \,\mathrm{d}\mathrm{vol}_t\,\mathrm{d}t
\,\, .
\end{aligned}
\end{equation}
By \eqref{eq:impHam*} and using that $t \mapsto t^2\mu(t)^{-1}$ is strictly increasing, we obtain
\begin{align}
3\int_0^{\tau}\int_M t^2\,|(D^t)^2v_t|_t^2\,\rho_t\,\mathrm{d}\mathrm{vol}_t\,\mathrm{d}t
&\leq 3\frac{\tau^2}{\mu(\tau)} \int_0^{\tau}\int_M
\mu(t)\,|(D^t)^2v_t|_t^2\,\rho_t\,\mathrm{d}\mathrm{vol}_t\,\mathrm{d}t \nonumber \\
&\leq \frac{3}{2\mu(\tau)}\,\tau^2 \int_M \log\left(\frac{A}{u_{\tau}}\right)\rho_{\tau}\,\mathrm{d}\mathrm{vol}_{\tau} \,\, . \label{eq:est(2)1}
\end{align}
By \eqref{eq:est2}, we obtain
\begin{equation} \label{eq:est(2)2}
2K_1\int_0^{\tau}\int_M t^2\,|\nabla^tv_t|_t^2\,\rho_t\,\mathrm{d}\mathrm{vol}_t\,\mathrm{d}t \leq 2K_1\tau^2\int_M \log\left(\frac{A}{u_{\tau}}\right)\rho_{\tau}\,\mathrm{d}\mathrm{vol}_{\tau} \,\, .
\end{equation}
By using \eqref{eq:impHam*}, H{\"o}lder's inequality and Young's inequality, along with the fact that $t\mapsto t^4\mu(t)^{-1}$ is strictly increasing,  we obtain
\begin{align}
2K_2&\int_0^{\tau}\int_M \frac{t^2}{(1+K_0t)^{\alpha}} |(D^t)^2v_t|_t \, \rho_t \,\mathrm{d}\mathrm{vol}_t\,\mathrm{d}t \nonumber \\
&\leq
2K_2\left(\int_0^{\tau} \frac{t^4}{\mu(t)(1+K_0t)^{2\alpha}}\,\mathrm{d}t\right)^{\frac12}
\left(\int_0^{\tau}\int_M \mu(t)\,|(D^t)^2v_t|_t^2 \, \rho_t \,\mathrm{d}\mathrm{vol}_t\,\mathrm{d}t\right)^{\frac12} \nonumber \\
&\leq \frac{K_2}{\sqrt{2}}\tau^2
\left(\frac{1}{\mu(\tau)}\frac{1-(1+K_0\tau)^{1-2\alpha}}{(2\alpha-1)K_0}
\right)^{\frac12} \left(1 +2\int_0^{\tau}\int_M \mu(t)\,|(D^t)^2v_t|_t^2 \, \rho_t \,\mathrm{d}\mathrm{vol}_t\,\mathrm{d}t\right) \nonumber \\
&\leq \frac{K_2}{\sqrt{2}}\,\tau^2
\left(
\frac{2}{2\alpha-1}\underbrace{\frac{1-(1+K_0\tau)^{1-2\alpha}}{1-e^{-2K_0\tau}}
}_{\leq1}\right)^{\frac12}
\left(1 +\int_M \log\left(\frac{A}{u_{\tau}}\right)\rho_{\tau}\,\mathrm{d}\mathrm{vol}_{\tau}\right) \nonumber \\
&\leq \frac{K_2}{\sqrt{2\alpha-1}}\,\tau^2\left(1 +\int_M \log\left(\frac{A}{u_{\tau}}\right)\rho_{\tau}\,\mathrm{d}\mathrm{vol}_{\tau}\right)
\,\, . \label{eq:est(2)3}
\end{align}
Finally, we have
\begin{align}
K_3&\int_0^{\tau}\int_M \frac{t^2}{(1+t)^{\alpha}}\,|\nabla^tv_t|_t \,\rho_t \,\mathrm{d}\mathrm{vol}_t\,\mathrm{d}t \nonumber \\
&\leq K_3\, \tau^2 \left(\int_0^{\tau} \frac{1}{(1+t)^{2\alpha}}\,\mathrm{d}t\right)^{\frac12}
\left(\int_0^{\tau}\int_M |\nabla^tv_t|_t^2 \, \rho_t \,\mathrm{d}\mathrm{vol}_t\,\mathrm{d}t\right)^{\frac12} \nonumber \\
&\leq \frac{K_3}{2}\tau^2 \left(\frac{1-(1+\tau)^{1-2\alpha}}{2\alpha-1}\right)^{\frac12} \left(1 +\int_0^{\tau}\int_M |\nabla^tv_t|_t^2 \, \rho_t \,\mathrm{d}\mathrm{vol}_t\,\mathrm{d}t\right) \nonumber \\
&\leq \frac{K_3}{2\sqrt{2\alpha-1}}\,\tau^2\left(1 +\int_M \log\left(\frac{A}{u_{\tau}}\right)\rho_{\tau}\,\mathrm{d}\mathrm{vol}_{\tau}\right) \,\, . \label{eq:est(2)5}
\end{align}
Then, by \eqref{eq:mass}, \eqref{eq:v-w-varphi}, \eqref{eq:estintvarphi2}, \eqref{eq:est(2)1}, \eqref{eq:est(2)2}, \eqref{eq:est(2)3}, and \eqref{eq:est(2)5}, we obtain
\begin{multline} \label{eq:estDeltalogu}
\int_M \tau^2\Delta_{\tau}\log(u_{\tau})\,\rho_{\tau}\,\mathrm{d}\mathrm{vol}_{\tau} \leq
\tau^2\left(\frac{n}{\tau} +\frac{K_2}{\sqrt{2\alpha-1}} +\frac{K_3}{2\sqrt{2\alpha-1}}\right) \\
+\tau^2\left(\frac{3}{2\mu(\tau)} +2K_1 +\frac{K_2}{\sqrt{2\alpha-1}} +\frac{K_3}{2\sqrt{2\alpha-1}}\right)\int_M \log\left(\frac{A}{u_{\tau}}\right)\rho_{\tau}\,\mathrm{d}\mathrm{vol}_{\tau} \,\, .
\end{multline}
Since
$$
\frac{\Delta_t u_t}{u_t} = \Delta_t \log(u_t)+|\nabla^t \log(u_t)|_t^2 \,\, ,
$$
\eqref{eq:Deltau/u} follows from \eqref{eq:HSZ} and \eqref{eq:estDeltalogu}.
\end{proof}

Theorem \ref{thm:intstatic} and Theorem \ref{thm:intRF} then follow from Theorem \ref{thm:2orderheat} by considering the cases where $g_t$ is a static metric and a Ricci flow solution, respectively.

\begin{proof}[Proof of Theorem \ref{thm:intstatic}]
Since $h_t = 0$, one can choose $K_2 = K_3 = 0$ in \eqref{eq:Deltau/u}.
\end{proof}

\begin{proof}[Proof of Theorem \ref{thm:intRF}]
Since $g_t$ is a Ricci flow solution, the second Bianchi identity implies that
$$
2\,\delta_th_t+\nabla^tH_t = 2\,\delta_t\mathrm{Ric}_t+\nabla^t\mathrm{scal}_t = 0
$$
and so one can choose $K_1 = K_2 = K$, $K_0 = K_3 = 0$ and $\alpha = \frac{3}{2}$ in \eqref{eq:Deltau/u}.
\end{proof}

\begin{remark} \label{rem:int-static2ord}
Under the hypotheses of Theorem \ref{thm:intstatic}, if $K_0=K_1=K$, then \eqref{eq:int-static2ord} and \eqref{eq:est} imply that
\begin{equation} \label{eq:upper'}
\frac{\Delta u_t}{u_t} \leq \frac{n}{t} +\left(\frac{5}{2t} +7K\right)\log\left(\frac{A}{u_t}\right) \quad \text{for all $t > 0$} \,\, .
\end{equation}
Notice that \cite[Theorem 1.1]{MR3472814} provides an estimate on the quantity $u_t^{-1} D^2u_t$, which leads to the following bound after taking the trace:
\begin{equation} \label{eq:estHZ2}
\frac{\Delta u_t}{u_t}\leq \frac{(5+B)n}{t}+\left(\frac{5}{t}n+Bn\right)\log\left(\frac{A}{u_t}\right) \,\, ,
\end{equation}
where $B \geq 0$ depends on the $L^\infty$-norms of the curvature tensor and the covariant derivative of the Ricci curvature. On the other hand, \eqref{eq:upper'} provides an explicit inequality depending only on a Ricci bound. The possibility of sharpening the constant of \eqref{eq:estHZ2} by looking at the trace bound was pointed out in \cite[p.\ 718, lines 4-5]{MR3472814}, but it has never appeared in print anywhere to our knowledge. Related upper Laplacian estimates appeared in \cite[Theorem E.36]{MR2365237} and \cite[Theorem 1.14-(a)]{MR3281843}.
\end{remark}

Finally, in the backward Ricci flow case, we obtain the following estimate.

\begin{corollary} \label{cor:HSZ2-backRF}
Let $(M^n,g_t)$ be a backward Ricci flow solution on a closed manifold, with $t \in [0,T]$, and assume that $0 \leq \mathrm{Ric}_t \leq Kg_t$ for some $K > 0$, for all $t \in [0,T]$. Let $u_t$ be a positive solution to the heat equation $\partial_t u_t = \Delta_t u_t$, for $t \in [0,T]$, with initial datum at $t=0$ satisfying $u_0 \leq A$, for some $A >0$. Then,
\begin{equation} \label{eq:IIorderRF}
\frac{\Delta_t u_t}{u_t} \leq \left(\frac{n}{t}+ \frac{K}{\sqrt{2}}\right) +\left(\frac{5}{2t} +\left(2+\frac{1}{\sqrt{2}}\right)K\right)\log\!\left(\frac{A}{u_t}\right) \quad \text{for all $t \in (0,T]$} \,\, .
\end{equation}
\end{corollary}

\begin{proof}
Since $h_t = -\mathrm{Ric}_t$, the second Bianchi identity implies that
$$
2\,\delta_th_t+\nabla^tH_t = -\big(2\,\delta_t\mathrm{Ric}_t+\nabla^t\mathrm{scal}_t\big) = 0 \,\, .
$$
Therefore, as in the proof of Theorem \ref{thm:intRF}, since $0 \leq \mathrm{Ric}_t \leq Kg_t$ for all $t \in [0,T]$, one can choose $K_1 = K_2 = K$, $K_0 = K_3 = 0$ and $\alpha = \frac{3}{2}$ in \eqref{eq:Deltau/u}.
\end{proof}

\end{document}